\documentclass[10pt,a4paper]{amsart}

\usepackage{graphicx}
\usepackage{epsfig}
\usepackage{epstopdf}
\usepackage{amsmath}
\usepackage{placeins}
\usepackage{amsfonts,amsmath, amsthm}
\usepackage{verbatim}
\usepackage{xfrac}

\usepackage{graphicx}
\usepackage[all]{xy}
\usepackage{caption}
\usepackage{subcaption}
\usepackage{float}
\usepackage{comment}
\usepackage{framed}

\theoremstyle{plain}
\newtheorem{theorem}{Theorem}

\newtheorem{rem}{Remark}[theorem]
\newtheorem{claim}{Claim}[theorem]
\newtheorem{lemma}{Lemma}[theorem]
\newtheorem{conj}{Conjecture}[theorem]

\newtheorem{corollary}{Corollary}[theorem]
\DeclareMathOperator{\So}{SO}

\DeclareMathOperator{\id}{id}

\DeclareMathOperator{\trace}{trace}
\DeclareMathOperator{\isom}{Isom}
\DeclareMathOperator{\dist}{dist}
\DeclareMathOperator{\spec}{spec}
\DeclareMathOperator{\Sim}{Sim}

\title{%
  {Convergence properties of a geometric mesh smoothing algorithm
}%
}

\author{%
  Dimitris Vartziotis}%
\address{TWT GmbH Science \& Innovation, Department for Mathematical Research \&
Services, Ernsthaldenstr. 17, 70565 Stuttgart, Germany;\newline
NIKI Ltd. Digital Engineering, Research Center, 205 Ethnikis Antistasis Street,
45500 Katsika, Ioannina, Greece
         }
\email{dimitris.vartziotis@nikitec.gr}
\author{Doris Bohnet}%
  \address{TWT GmbH Science \& Innovation, Department for Mathematical Research \&
Services, Ernsthaldenstr. 17, 70565 Stuttgart, Germany}
\email{doris.bohnet@twt-gmbh.de}

\begin{document}
\maketitle

\begin{abstract}
We describe a simple geometric transformation of triangles which leads to an efficient and effective algorithm to smooth triangle and tetrahedral meshes. Our focus lies on the convergence properties of this algorithm: we prove the effectivity for some planar triangle meshes and further introduce dynamical methods to study the dynamics of the algorithm which may be used for any kind of algorithm based on a geometric transformation. 
\end{abstract}
\maketitle

\section{Introduction}\label{s.first}
\subsection{Preliminary remarks}
The finite element method is the standard instrument to simulate the behavior of solid bodies or fluids in engineering and physics. The first preparatory step of this method is the discretization of the underlying domain into finitely many elements which could be easily described by parameters, i.e. surfaces are mostly approximated by triangles, quatrilaterals or parallelograms. Because of the design process in modern engineering, an initial mesh for the domain is often given, and the next important step is the preprocessing of this mesh to obtain a good base for the application of the finite element method. As the requirements for simulation results are more and more strict and real time simulations and simulations on evolving objects present new challenges, a fast, reliable and preferably automatic mesh preprocessing is an important link in the simulation process.  \\
Not surprisingly, there is a wide variety of methods at hand to improve the mesh quality of a given mesh, see e.g. the surveys \cite{FreyGeorge2000} and \cite{Carey1998}. One can identify two main approaches: 

\begin{description}
\item[Geometry-based] A geometric smoothing method changes directly the geometry of the mesh, that is, it relocates the nodes. A popular example is the Laplacian smoothing which maps every node to the arithmetic mean of its neighboring nodes (see e.g. \cite{H76} or \cite{Field1988}). There are also methods that change the topology of the mesh by deleting small elements or subdividing large ones. All these methods have in common that they are usually quickly implemented and very fast; additionally, they can be often straightforwardly combined with techniques of parallel computing. The main disadvantage is their heuristic character so that the convergence of an algorithm is mostly only empirically, but not theoretically assured. Consequently, they are sometimes combined with optimizational approaches as in the early work \cite{Freitag1997}.    
\item[Optimization-based] The principal idea behind any optimizational approach is to define a function on the set of meshes which represents the quality of a mesh, and to find the maximum of this function by usual numerical optimization, e.g. gradient methods (see \cite{FreitagKnupp1999} and following articles by these authors). The main advantage of these methods is clearly that they lead to a mesh of higher quality, but in the case of a non-convex quality function, it can usually not be assured that the transformed mesh corresponds to the global and not only local optimum. Also, the computation is usually costly with regard to runtime and storage. On the other hand, new advances for fast and robust solutions of optimization problems could be used as e.g. evolutionary optimization algorithms (see e.g. \cite{YK09}, \cite{HK03}). 
\end{description}
In this article we present a geometric approach for mesh smoothing which consists of a simple geometric transformation of every element of a mesh which do not affect the topology. Here we only consider triangle and tetrahedral meshes. This ansatz is similar to the GETMe algorithm introduced for triangle and tetrahedral meshes in \cite{VAGW08,VWS09} and proved to be element-wise effective in \cite{VH14}. But while the serial of these articles mainly focus on the numerical results and the improvement of runtime and performance by adjusting the algorithm, we study the mathematics underlying our presented geometric method and prove that any not too distorted planar mesh of triangles converges to the best possible mesh for the given mesh topology, that is, the difference between the normalized distances from the vertices to the centroid for all triangles is the smallest possible. 
This point - the mathematical discussion of convergence properties and the application of dynamical methods - is surely the main achievement of the present article. For completeness and as motivation for a future application, we shortly discuss the performance of our method as a smoothing algorithm, but we do not explore the practical aspects of our algorithm in detail.
\subsubsection{Organization of the article} In Section~\ref{s.trans} we describe the discrete geometric transformation of a triangle element on which the smoothing algorithm is based and discuss its mathematical properties. In the forthcoming Section~\ref{s.convergence} we derive the smoothing algorithm from this transformation for a triangle mesh and prove its convergence for some particular triangle meshes, i.e. if the transformation is iteratively applied, the mesh converges to the best possible mesh for a given mesh topology. At the end -- in Section~\ref{s.numerical} -- we briefly discuss the numerical results.

\section{The geometric triangle transformation}\label{s.trans}
Before we start with the rigorous mathematical description we motivate the geometric transformation, the subject of this article, and its regularizing mechanism by the following observations which have their offspring in \cite{V13}:
\subsection{Introductory observations}  
\subsubsection{Imitating the rotational symmetry group action of the triangle}
The symmetry group of a regular triangle $\Delta=(z_0,z_1,z_2)$, $z_i \in \mathbb{C}$, is the dihedral group $D_{3}$ which is generated by a reflection and a rotation by $\frac{2\pi}{3}$ around the circumcenter $c$ of the triangle. Consider the rotation: if the circumcenter lies in the origin, the rotational element then acts on the triangle by mapping the vector $z_{i-1}$ onto the vector $z_i$ for $i\in\mathbb{Z}_3$. \\
If the triangle is not equilateral, we can take the centroid, that is, the arithmetic mean of the three nodes, instead of the circumcenter and imitate the rotation by still mapping the vector $z_{i-1}$ onto $\frac{|z_{i-1}|}{|z_i|}z_{i}$ such that the resulting vector has length equal to $z_{i-1}$ but points into the direction of $z_i$, i.e. we \emph{rotate} the vector $z_{i-1}$ around $c$.\\

Sure, this rotation around the centroid is neither $3$-periodic nor isometric, but this action, if iterated, converges to the classical rotation by $\frac{2\pi}{3}$ because the centroid converges to the circumcenter and the distances from the centroid to the vertices become equal. This will be shown in Subsection~\ref{s.proof} below.
\begin{figure}[htbp]

\includegraphics[width=\textwidth]{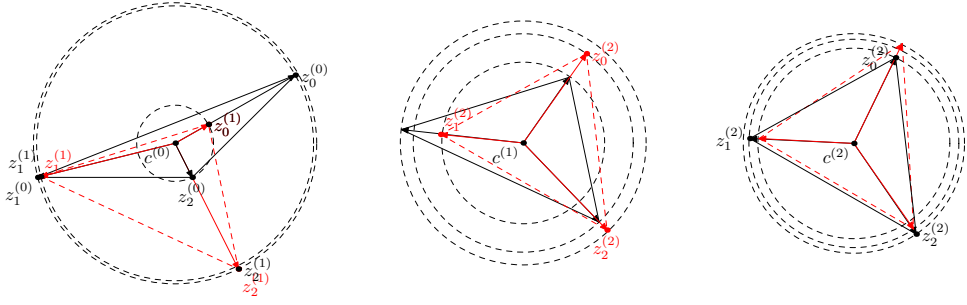}

\caption{The first three iterations of the geometric element transformation represented as rotations: observe how the circle radii approach and the centroid moves to the circumcenter.}
\label{fig:triangle_rotated}
\end{figure}

\subsubsection{First intuitive explanation of the mechanism}
Consider the centroid $c$. We start at the origin $c^{(0)}=0$. After the first iteration we get
$c^{(1)}=\frac{1}{3}\left(\frac{\left|z_2\right|}{\left|z_0\right|}z_0 + \frac{\left|z_0\right|}{\left|z_1\right|}z_1 + \frac{\left|z_1\right|}{\left|z_2\right|}z_2\right).$
Just change it a bit setting $r_i:=\frac{\left|z_{i-1}\right|}{\left|z_i\right|}$ to obtain
$\tilde{c}^{(1)}=\frac{1}{r_0+r_1+r_2}\left(r_0z_0 + r_1z_1 + r_2z_2\right)$.
Observe that by the inequality of the geometric and arithmetic mean we always have  $3=3(\sqrt[3]{r_0r_1r_2})\leq r_0+r_1+r_2$.
The new $\tilde{c}$ is a \emph{weighted arithmetic mean} of the vertices where the weight $r_i:=\frac{\left|z_{i-1}\right|}{\left|z_i\right|}$ is greater the smaller the distance $\left|z_i\right|$ compared to $\left|z_{i-1}\right|$. Accordingly, $c$ is moved more than the average into the direction of $z_i$ if $z_i$ was much closer to $c$ then $z_{i-1}$. The new distance from $c^{(1)}$ to $\frac{\left|z_{i-1}\right|}{\left|z_i\right|}z_i$ is then less than $\left|z_{i-1}\right|$. On the other hand, if $z_i$ was far from $c$ compared to $z_{i-1}$, $c$ is moved less than the average into the direction of $z_i$, and the new distance from $c^{(1)}$ to $\frac{\left|z_{i-1}\right|}{\left|z_i\right|}z_i$ is greater than $\left|z_{i-1}\right|$. Consequently, due to the controlling weights the maximum distance from a vertex to the centroid lessens while the minimal distance augments such that the distances become equal in the long run. In other words, the centroid moves to the circumcenter of the triangle. It is obvious that $c$ stops moving if and only if the distance to each vertex is equal or equivalently, if and only if $c$ is the circumcenter.\\
In Section~\ref{s.proof} we give a rigorous proof that this transformation iteratively applied to any non-degenerate triangle makes it equilateral.

\subsection{The geometric element transformation}
\subsubsection{Description of the geometric element transformation}

Let $\Delta=(x_0,x_1,x_2)$ with $x_i \in \mathbb{R}^2$ or $x_i \in \mathbb{R}^3$ for $i=0,1,2$ be a triangle in the euclidean space where the vertices are denoted counter clockwisely. Denote by $c=\frac{1}{3}(x_0 + x_1 + x_2)$ the centroid of the triangle. The transformation works then as following: 
\begin{align*}\Delta_{new}&=(x_{0,new},x_{1,new},x_{2,new}),\\  
x_{i,new}&=\frac{\left\|x_{i-1} - c\right\|_2}{\left\|x_i - c\right\|_2}(x_i - c) + c \quad \mbox{for}\; i\in\mathbb{Z}_3.
\end{align*}  
To keep the centroid fixed we move the centroid $c_{new}$ of the transformed triangle $\Delta_{new}$ back into the old centroid $c$: 
$x_{i,new}=x_{i,new}-c_{new} + c, \; i=0,1,2.$
Combining these two steps into one we get for $i\in\mathbb{Z}_3$ with $r_i:=\left\|x_{i-1}-c\right\|\left\|x_i-c\right\|^{-1}$:
\begin{framed}
\begin{equation}\label{t.getme}x_{i,new}=\frac{2}{3}r_i(x_i - c) - \frac{1}{3}r_{i+1}(x_{i+1} - c)- \frac{1}{3}r_{i-1}(x_{i-1} - c) + c. \end{equation}
\end{framed}
This is the whole, very simple geometric transformation which can be directly implemented into any mathematical software, e.g. Matlab. 
\subsubsection{Formal proof of element-wise convergence}\label{s.proof}
For the proof that any non-degenerate triangle converges under Transformation~(\ref{t.getme}) to a equilateral triangle we can supppose that the triangle lies in the euclidean plane $\mathbb{R}^2$ which we identify -- for simplifying notations -- with $\mathbb{C}$: 
let $\Delta=(z_0^{(0)},z_1^{(0)},z_2^{(0)})$ with $z_i^{(0)} \in \mathbb{C}$ be an arbitrary triangle with $z_i^{(0)} \neq z_j^{(0)}$ for $i \neq j$. Denote by $c=\frac{1}{3}(z_0 + z_1 + z_2)$ the centroid of the triangle. Without loss of generality, we assume that the centroid $c$ lies in the origin, that is, $c=0$. For $n \in \mathbb{N}$ denote by $r_i^{(n)}$ the ratio $\frac{|z_{i-1}^{(n)}|}{|z^{(n)}_{i}|}$ with $i \in \mathbb{Z}_3$. Then Transformation~(\ref{t.getme}) becomes the following transformation, recursively defined for $n \geq 1$ and $i \in \mathbb{Z}_3$: 
\begin{equation}\label{e.transformation}
z_i^{(n)}=\frac{2}{3}r_{i}^{(n-1)}z_i^{(n-1)} -\frac{1}{3}r_{i+1}^{(n-1)}z_{i+1}^{(n-1)}-\frac{1}{3}r_{i+2}^{(n-1)}z_{i+2}^{(n-1)} 
\end{equation}
We prove that the ratio $r_i^{(n)}$ for $i=0,1,2$ converges to $1$. This implies that the centroid converges to the circumcenter and the distance of the vertices to the centroid gets constant. 
\begin{theorem}\label{t.convergence}
With the notations above, we have 
$\lim_{n \rightarrow \infty} r_i^{(n)}= 1$
for $i=0,1,2$, i.e. the ratio of the distances from the vertices to the centroid converges to $1$.  
\end{theorem}
Before we start with the proof of Theorem~\ref{t.convergence} we show in the next two preliminary lemmata that the maximal distance $\max_i|z_i^{(n)}|$ from a vertex $z_i^{(n)}$ to the centroid $c$ is a strictly decreasing sequence, and symmetrically, that the minimal distance from a vertex to the centroid is a strictly increasing sequence.  
\begin{lemma}\label{l.max}
For $n \geq 0$ we have $\max_{i=0}^2 |z_i^{(n+1)}| < \max_{i=0}^{2}|z_i^{(n)}|$. 
\end{lemma}
\begin{proof}[Proof of Lemma~\ref{l.max}]
We prove this Lemma by a simple estimation. Without loss of generality we assume that $\max_{i=0}^2|z_i^{(n+1)}| = |z_0^{(n+1)}|$. We have
\begin{align*}
|z_0^{(n+1)}| &= |\frac{2}{3} r_0^{(n)} z_0^{(n)} - \frac{1}{3}r_1^{(n)}z_1^{(n)}- \frac{1}{3}r_2^{(n)}z_2^{(n)}|\; \mbox{utilizing}\; z_0^{(n)} + z_1^{(n)}+z_2^{(n)}=0\\
&=|-\frac{2}{3}r_0^{(n)}(z_1^{(n)} + z_2^{(n)})- \frac{1}{3}(r_1^{(n)}z_1^{(n)} + r_2^{(n)}z_2^{(n)})|\\
&< \max_{i}r_i^{(n)} |z_0^{(n)}|
\end{align*}
If $r_0^{(n)}=\max_i r_i^{(n)}$ the proof is easily finished by 
$$|z_0^{(n+1)}| < r_0|z_0^{(n)}|=|z_2^{(n)}| \leq \max_i |z_i^{(n)}|.$$
Otherwise, assume that $r_1^{(n)}$  is maximal (the case that $r_2^{(n)}$ is maximal works analogously). Then we substitute in the equation above $z_1^{(n)}=-z_0^{(n)}-z_2^{(n)}$ and we get: 
\begin{align*}
|z_0^{(n+1)}| &= |(\frac{2}{3}r_0^{(n)} + \frac{1}{3}r_1^{(n)})z_0^{(n)}+ (\frac{1}{3}r_1^{(n)}-\frac{1}{3}r_2^{(n)}) z_2^{(n)}|\; r_1^{(n)}=\max r_i^{(n)}.\\
& < r_1^{(n)}|z_1^{(n)}|= |z_0^{(n)}|\;\leq \; \max |z_i^{(n)}|.
\end{align*}
\end{proof}
Now we prove in an analogous way that the sequence of minima is monotonically increasing:
\begin{lemma}\label{l.min}
For $n \geq 0$ we have $\min_{i=0}^2 |z_i^{(n)}| < \min_{i=0}^2 |z_i^{(n+1)}|$. 
\end{lemma}
\begin{proof}[Proof of Lemma~\ref{l.min}]
Assume without loss of generality that $|z_0^{(n+1)}|=\min |z_i^{(n+1)}|$. We have the following estimate: 
\begin{align}\label{e.estimate1}
|z_0^{(n+1)}| &= |\frac{2}{3}r_0^{(n)}z_0^{(n)} - \frac{1}{3}\left(r_1^{(n)}z_1^{(n)} + r_2^{(n)}z_2^{(n)}\right)|\; \mbox{substituting}\; z_0^{(n)}=-z_1^{(n)}-z_2^{(n)}\nonumber\\
&=|- \frac{2}{3}r_0^{(n)}\left(z_1^{(n)} + z_2^{(n)}\right) - \frac{1}{3}\left(r_1^{(n)}z_1^{(n)} + r_2^{(n)}z_2^{(n)}\right)|\nonumber\\
& > \left(\frac{2}{3}r_0^{(n)} + \frac{1}{3}\min_{i=1,2}r_i^{(n)}\right)|z_0^{(n)}|
\end{align}
Consequently, we have to show that (\ref{e.estimate1}) is greater than $\min_i|z_i^{(n)}|$: If $r_0^{(n)}$ is minimal, we easily get
$$  \left(\frac{2}{3}r_0^{(n)} + \frac{1}{3}\min_{i=1,2}r_i^{(n)}\right) \geq r_0^{(n)}\; \Rightarrow\; |z_0^{(n+1)}| > r_0|z_0^{(n)}|=|z_2^{(n)}| \geq \min |z_i^{(n)}|.$$
Otherwise, if $r_0^{(n)}$ is maximal, we certainly have -- utilizing the inequality of arithmetic and geometric mean
$$\left(\frac{2}{3}r_0^{(n)} + \frac{1}{3}\min_{i=1,2}r_i^{(n)}\right)\geq\left(\frac{1}{3}r_0^{(n)} + \frac{1}{3}r_1^{(n)} + \frac{1}{3}r_2^{(n)}\right)\geq \sqrt[3]{r_0^{(n)}r_1^{(n)}r_2^{(n)}} = 1, $$
finishing the proof for this case due to $|z_0^{(n+1)}| > |z_0^{(n)}|$. In the last case, if $r_0^{(n)}$ is neither maximal nor minimal, either $r_1^{(n)}$ or $r_2^{(n)}$ is maximal. Assume without loss of generality that $r_1^{(n)}$ is maximal: 
\begin{align*}
|z_0^{(n+1)}| &> \left(\frac{2}{3}r_0^{(n)} + \frac{1}{3}r_2^{(n)}\right)|z_0^{(n)}|\geq \sqrt[3]{r_0^{(n)}r_0^{(n)}r_2^{(n)}}|z_0^{(n)}| \\
&= \sqrt[3]{r_1^{(n)}r_1^{(n)}r_0^{(n)}}|z_1^{(n)}|\quad \mbox{utilizing}\; |z_0^{(n)}|=r_1^{(n)}|z_1^{(n)}|\; \mbox{and}\;r_0r_1r_2=1\\
&= \sqrt[3]{\frac{r_1^{(n)}}{r_2^{(n)}}}|z_1^{(n)}| \geq |z_1^{(n)}| \geq \min_i |z_i^{(n)}|,
\end{align*} 
finishing the proof.
\end{proof}
With the help of these two lemmas we can directly conclude Theorem~\ref{t.convergence}:
\begin{proof}[Proof of Theorem~\ref{t.convergence}]
For $i\in\mathbb{Z}_3$ consider the sequence $\left(r_i^{(n)}=\frac{|z_{i-1}^{(n)}|}{|z_i^{(n)}|}\right)_{n \geq 0}$. We have for $n\geq 0$ the following bounds from below and above: 
\begin{equation}\label{e.estimate2}
\frac{\min_i |z_i^{(n)}|}{\max_i |z_i^{(n)}|}\leq r_i^{(n)} \leq \frac{\max_i |z_i^{(n)}|}{\min_i |z_i^{(n)}|}.
\end{equation}
According to Lemmas~\ref{l.max} and \ref{l.min}, the sequence $\left(\frac{\max_i |z_i^{(n)}|}{\min_i |z_i^{(n)}|}\right)_{n\geq 0}$ is a strictly decreasing sequence bounded from below by $1$, so it converges to $1$; in the same way, the sequence $\left(\frac{\min_i |z_i^{(n)}|}{\max_i |z_i^{(n)}|}\right)_{n\geq 0}$ is a strictly increasing sequence bounded from above by $1$, so it also converges to $1$. 
These two results combine with (\ref{e.estimate2}) to $\lim_{n \rightarrow \infty} r_i^{(n)} =1$ for $i=0,1,2$ finishing the proof. 
\end{proof}
Theorem~\ref{t.convergence} directly gives us the required result for the geometric element transformation where we assume that $\Delta^n$ is non-degenerate, that is, the vertices are pairwise disjoint:
\begin{corollary}[Elementwise convergence]\label{c.convergence}
The triangle $\Delta^n=(z_0^{(n)}, z_1^{(n)}, z_2^{(n)})$ converges for $n \rightarrow \infty$ to an equilateral triangle. 
\end{corollary}
\begin{proof}[Proof of Corollary~\ref{c.convergence}]
As $r_i^{(n)}=|\frac{z_{i-1}^{(n)}}{z_i^{(n)}}|$ converges to $1$ with $n \rightarrow \infty$, we get that $\lim |z_i^{(n)}| = \lim |z_{j}^{(n)}|$ for $i,j=0,1,2$. Therefore, the distances from the vertices to the centroid $c$ become equal, so that $c$ becomes the circumcenter, and the triangle equilateral.  
\end{proof}

\section{Convergence of the smoothing algorithm for triangle meshes}\label{s.convergence}
Transformation~(\ref{t.getme}) can be used to transform a mesh of triangles by combining it at every vertex with taking the barycenter. We give the precise definition of the considered mesh transformation below after specifying in detail our setting. \\ 
We prove in this section that any triangle mesh which does not contain too pathological triangles converges under the transformation to a mesh of triangles as regular as possible. Let us make precise our setting: 
\subsection{Preliminary notations:}
Let $\Sigma=\left\{0,\dots,N-1\right\}$ be a finite set of symbols. Let \newline$C=\left\{\Delta_i=(i_0,i_1,i_2)\in \Sigma^3\,\big|\, i=0,\dots,n-1\right\}$ be a finite set of triples of symbols.  We call the set $C$ a \emph{connectivity} iff for any pair $\Delta_i, \Delta_j\in C$ there exist $k\leq n-1$ and a finite sequence $\Delta_0,\dots,\Delta_k \in C$ such that $\Delta_0=\Delta_i$ and $\Delta_k=\Delta_j$ and for $m=0,\dots,k-1$ the triples $\Delta_m$ and $\Delta_{m+1}$ have exactly two symbols in common. 
Let $M_C: \Sigma \rightarrow \mathbb{R}^2, k\mapsto x_k$ be an injective map. We call $M_C$ a \emph{triangle mesh with connectivity C} iff for any $i,j=0,\dots,n-1$, $i\neq j$ the triangles defined by $M_C(\Delta_i):=(M_C(i_0),M_C(i_1),M_C(i_2))$ and $M_C(\Delta_j)$, counted counter clockwisely, are non-degenerate and have disjoint interior. Let denote by $X_C$ the set of meshes $M_C$ with connectivity $C$. 
\begin{rem} As a consequence of the definition of connectivity, the set $\bigcup_{i=0}^{n-1}M_C(\Delta_i)$ is arcwise connected.
\end{rem} 
\subsection{Definition of the mesh transformation}
We define a triangle mesh transformation in two steps. First, for any $i=0,\dots,n-1$ we define for the triangle $M_C(\Delta_i)=(x_{i_0},x_{i_1},x_{i_2})$ with centroid $c_i$ the triangle transformation (as defined above) by  
\begin{align}\label{e.triangle_transformation}
&\theta_i:\mathbb{R}^6 \rightarrow \mathbb{R}^6\;\nonumber\\
&x_i:=(x_{i_0},x_{i_1},x_{i_2}) \mapsto (\theta_{i_0}(x_{i}),\theta_{i_1}(x_{i}),\theta_{i_2}(x_{i})),\nonumber\\
&\theta_{i_0}(x_{i})= \frac{2}{3}\frac{\left\|x_{i_2}-c_i\right\|}{\left\|x_{i_0}-c_i\right\|}(x_{i_0}-c_i) -\frac{1}{3}\frac{\left\|x_{i_0}-c_i\right\|}{\left\|x_{i_1}-c_i\right\|}(x_{i_1}-c_i)-\frac{1}{3}\frac{\left\|x_{i_1}-c_i\right\|}{\left\|x_{i_2}-c_i\right\|}(x_{i_2}-c_i) + c_i\nonumber\\
&\theta_{i_1}(x_{i})= \frac{2}{3}\frac{\left\|x_{i_0}-c_i\right\|}{\left\|x_{i_1}-c_i\right\|}(x_{i_1}-c_i) -\frac{1}{3}\frac{\left\|x_{i_1}-c_i\right\|}{\left\|x_{i_2}-c_i\right\|}(x_{i_2}-c_i)-\frac{1}{3}\frac{\left\|x_{i_2}-c_i\right\|}{\left\|x_{i_0}-c_i\right\|}(x_{i_0}-c_i)+c_i\nonumber\\
&\theta_{i_2}(x_{i})= \frac{2}{3}\frac{\left\|x_{i_1}-c_i\right\|}{\left\|x_{i_2}-c_i\right\|}(x_{i_2}-c_i) -\frac{1}{3}\frac{\left\|x_{i_2}-c_i\right\|}{\left\|x_{i_0}-c_i\right\|}(x_{i_0}-c_i)-\frac{1}{3}\frac{\left\|x_{i_0}-c_i\right\|}{\left\|x_{i_1}-c_i\right\|}(x_{i_1}-c_i)+c_i. 
\end{align}
We can now define the mesh transformation under consideration.  
For $k=0,\dots,N-1$ let $\Sigma_k$ denote the set of indices of the adjacent triangles at $x_k$ and we define the map 
\begin{align}\label{e.mesh_transformation}
\Theta:X_C \subset \mathbb{R}^{2N}&\rightarrow \mathbb{R}^{2N}\nonumber\\
x=(x_0,\dots,x_{N-1})&\mapsto (\Theta_0(x),\dots,\Theta_{N-1}(x)), \; x_k, \Theta_k(x) \in \mathbb{R}^2\nonumber, \\
 \mbox{with}\;&\Theta_k(x)=\frac{1}{|\Sigma_k|}\sum_{m \in \Sigma_k}t_{m_{j(m)}}(x_k),
\end{align}
where $j(m) \in {0,1,2}$ denotes the index of the vertex $x_k$ inside the triangle numbered by $m$, so $\theta_{m_{j(m)}}:\mathbb{R}^2 \rightarrow \mathbb{R}^2$.  
\begin{rem}
The map $\Theta$ is clearly well-defined as a map from $X_C$ to $\mathbb{R}^{2N}$, but not as a map to $X_C$: let $M_C(\Sigma)=(x_0,\dots,x_{N-1}) \in \mathbb{R}^{2N}$ be a triangle mesh with connectivity $C$. Then $\Theta(M_C(\Sigma))$ is not necessarily a triangle mesh $M'_C$. It could happen that the interiors of two triangles $\Theta(\Delta_i)$, $\Theta(\Delta_j)$ are no longer disjoint. 
\end{rem}
On the other hand note that a mesh of equilateral triangles is fixed under $\Theta$, such that we can prove the following lemma where we call  \emph{distortion} of a triangle the ratio of the shortest by the longest edge length of a triangle:
\begin{lemma}{Well-definedness of $\Theta$}\label{l.welldefinedness}
Let $\Theta$ be defined as above. Then there exists $0< \delta < 1$ such that for any triangle mesh $M_C$ whose distortion of triangles is bounded from below by $\delta$ the image $\Theta(M_C)$ is a triangle mesh $M'_C$. 
\end{lemma} 
We postpone the proof to Subsection~\ref{s.similarity}. 

\subsection{Similarity group action and equivariance of $\Theta$}\label{s.similarity}
Denote by $\Sim(\mathbb{R}^2)$ the four-dimensional group of similarities of $\mathbb{R}^2$ composed by the one-dimensional group $\mathbb{R}^+$ of scaling and the three-dimensional group $\isom(\mathbb{R}^2)$ of isometries. This groups naturally acts on the set of triangles by 
$$\Sim(\mathbb{R}^2) \times (\mathbb{R}^2)^3 \rightarrow (\mathbb{R}^2)^3; (g, x)\mapsto g.x=(g.x_0,g.x_1,g.x_2)$$
for any triangle $x=(x_0,x_1,x_2) \in \mathbb{R}^6$ where there exists $A \in \So(2,\mathbb{R}^2)$, $\theta \in \mathbb{R}^2$ and $\lambda \in \mathbb{R}^+$ such that $g.x_i=\lambda(Ax_i + \theta)$ for $i=0,1,2$.
\begin{rem}
Two triangles $x,y$ are \emph{similar} iff there exists $g \in \Sim(\mathbb{R}^2)$ such that $g.x=y$. \\
\end{rem}
One can easily prove that the group $\Sim(\mathbb{R}^2)$ acts freely on the set of triangles:
\begin{lemma}
The group action as defined above is free.
\end{lemma}
\begin{proof}
Let $A \in \So(2,\mathbb{R}), \theta \in \mathbb{R}^2$ and $\lambda \in \mathbb{R}^+$ such that $\lambda(Ax + \theta)=x$ for a triangle $x=(x_0,x_1,x_2) \in \mathbb{R}^6$. This implies immediately that $\lambda=1$. So we have $Ax_0 + \theta = x_0$ and $Ax_1 + \theta=x_1$. We conclude that $A(x_0-x_1)=x_0-x_1$. This means that $x_0-x_1$ is the eigenvector to an eigenvalue $1$ of $A$.  So we can conclude that $A$ is the identity. Consequently, we get $\theta=0$ finishing the proof.
\end{proof} 
The action defined above can be straightforwardly generalized to the set $X_C$ of meshes with connectivity $C$ by 
$$(g,x) \in \Sim(\mathbb{R}^2) \times (\mathbb{R}^2)^N\,\mapsto \, (g.x_0, g.x_1,\dots,g.x_{N-1})$$ 
for any mesh $M_C(\Sigma)=:x \in (\mathbb{R}^2)^N$. This action is certainly free as well. 
\begin{rem}
One could think of defining the similarity group action on a mesh separately on every triangle. But in fact, the connectivity as defined above forces that the same group element acts simultaneously on each triangle of the mesh. So the group action defined above is the only one in accordance with the given definition of a mesh.   
\end{rem}
As a consequence we can list the following properties of the group action: 
\begin{enumerate}  
\item Every $\Sim(\mathbb{R}^2)$-orbit is a four-dimensional smooth submanifold in $\mathbb{R}^{2N}$.
\item One computes immediately that the mesh transformation $\Theta$ is \emph{equivariant} under the group action of $\Sim(\mathbb{R}^2)$, that is 
\begin{enumerate}
\item For every $M_C$ inside the domain of $\Theta$ the image $\Theta(M_C)$ lies as well in the domain.  
\item $$\Theta(g.M_C) = g.\Theta(M_C) \quad \mbox{for any} \;M_C \in X_C \; \mbox{and}\; g \in \Sim(\mathbb{R}^2).$$
\end{enumerate} 
(see e.g. \cite{F80} where important properties for equivariant dynamical systems are proved).
\item For any $h \in \Sim(\mathbb{R}^2)$ one computes for the Jacobian matrix of $\Theta$ for any $M_C \in X_C$ 
$$\Theta=h^{-1} \circ \Theta \circ h \; \Leftarrow \; D\Theta_{M_C}=Dh^{-1} \circ D\Theta_{h.M_C} \circ Dh,$$
and as $h,h^{-1}$ are linear maps one gets
$$ D\Theta_{M_C}=h^{-1} \circ D\Theta_{h.M_C} \circ h.$$
\end{enumerate}
Thanks to these properties we can reduce the question of global convergence to the following: 
Let $M_C$ be a fixed point of $\Theta$, that is $\Theta(M_C)=M_C$, then the whole group orbit $\Lambda:=\Sim(\mathbb{R}^2).M_C$ is fixed and $\Lambda$ is consequently a four-dimensional submanifold of fixed points. \\
Now we can prove Lemma~\ref{l.welldefinedness}:
\begin{proof}[Proof of Lemma~\ref{l.welldefinedness}]
Let $M_{eq}$ be an equilateral mesh, then we have $\Theta(M_{eq})=M_{eq}$. On the other hand, the domain of $\Theta$ is clearly an open subset of $X_C$. By the continuity of $\Theta$, there exists an open set $\mathcal{U}$ of triangles sufficiently close to $M_{eq}$ such that $\Theta$ is well-defined on $\mathcal{U}$. The equivariance of $\Theta$ implies that $\Theta$ is well-defined on the group orbit $\Sim(\mathbb{R}^2).\mathcal{U}$ of $\mathcal{U}$ which contains all meshes whose distortion of triangles is bounded by some $ 0<\delta < 1$ which ends the proof.   
\end{proof}
To study the convergence in a neighborhood of the fixed point $M_C$ it is enough to study the dynamics of $\Theta$ in a neighborhood of $\Lambda$ thanks to the following lemma: 
\begin{lemma}\label{l.convergence}
Let $x=M_C\in\mathbb{R}^{2N}$ be a fixed point of $\Theta$ and $\Lambda:=\Sim(\mathbb{R}^2).M_C$ its group orbit. If there exists a $D\Theta$-invariant decomposition of the tangent bundle at $\Lambda$ $$T\mathbb{R}^{2N}|_{\Lambda}=T\Lambda \oplus E^s$$
such that $\left\|D\Theta|_{E^s}\right\| < 1$, then there exists a unique family $\mathcal{F}^s$ of injectively $C^r$-immersed submanifolds $\mathcal{F}^s(x)$ such that $x \in \mathcal{F}^s(x)$ and $\mathcal{F}^s(x)$ is tangent to $E^s_x$ at every $x \in \Lambda$. This family is $\Theta$-invariant, that is, $\Theta(\mathcal{F}^s(x))= \mathcal{F}^s(\Theta(x))$, and the manifolds $\mathcal{F}^s(x)$ are uniformly contracted by some iterate of $\Theta$. \\
That family actually forms a foliation of a neighborhood of $\Lambda$.  
\end{lemma} 
This Lemma is an immediate application of the invariant manifold theorem by Hirsch,Pugh and Shub, cited and proved for example in \cite[Th.B7, p.293]{BDV}.
The spectrum $\spec(D\Theta|_{\Lambda})$ tangent to the group orbit contains four eigenvalues equal to $1$. Consequently, we have the following direct corollary of Lemma~\ref{l.convergence}:
\begin{corollary}\label{c.convergence}
Let $M_C\in\mathbb{R}^{2N}$ be a fixed point of $\Theta$ and $\Lambda:=\Sim(\mathbb{R}^2).M_C$ its group orbit. If every eigenvalue of the Jacobian matrix $D\Theta$ which is not contained in $\spec(D\Theta|_{\Lambda})$ has an absolute value strictly smaller than $1$, then $\Lambda$ is an attractor and $\Theta^n(M)$ converges uniformly at exponential rate to \textbf{one} point in $\Lambda$ for $n \rightarrow \infty$ and for any triangle mesh $M$ sufficiently close to $\Lambda$.
\end{corollary}
So as a consequence of this corollary, it is enough to study the spectrum of $D\Theta$ at a fixed point and to prove that the absolute value of all eigenvalues except from four is strictly smaller than one. Nevertheless, this is still a difficult task as it will become obvious in the following. We start with the easiest cases gaining more and more complexity:
\newpage
\subsection{Convergence for particular cases}
\subsubsection{Case 1: A single triangle}
We start with the easiest case of a mesh which consists of a single triangle, so in fact, we study the global convergence of the previously defined triangle transformation $\theta:\mathbb{R}^6 \rightarrow \mathbb{R}^6$ on a triangle $x=(x_0,x_1,x_2) \in (\mathbb{R}^2)^3$ in more details and using the new setting above. 
Let $x_{eq}\in \mathbb{R}^6$ be an equilateral triangle, then $\theta(x_{eq})=x_{eq}$. The equivariance of $\theta$ under the group of similarities provokes that $$\theta(\Lambda_{eq}) = \Lambda_{eq},\;\mbox{where}\; \Lambda_{eq}=\left\{x \in \mathbb{R}^6 \;\big|\; g \in \Sim(\mathbb{R}^2): \; x=g.x_{eq},\; \right\}.$$  
is a $4$-dimensional submanifold of $\mathbb{R}^6$ which is $\theta$-invariant, that is, $\theta(\Lambda_{eq}) \subset \Lambda_{eq}$. Following Corollary~\ref{c.convergence} we compute the derivative $D\theta_{x_{eq}}$ of $\theta$ at $x_{eq} \in \Lambda_{eq}$. The Jacobian matrix is the same matrix $J$ for any $x_{eq} \in \Lambda_{eq}$: 
\begin{align}\label{e.jacobian}
Dt(x_{eq}) &= \begin{pmatrix}A & B & C \\C & A&B\\B & C & A\end{pmatrix}=:J\quad \mbox{where}\nonumber \\
A&=\begin{pmatrix}\frac{3}{4}& -\frac{1}{4\sqrt{3}}\\\frac{1}{4\sqrt{3}}& \frac{3}{4}\end{pmatrix},\; B=\begin{pmatrix}\frac{1}{4}& -\frac{1}{4\sqrt{3}}\\ \frac{1}{4\sqrt{3}}&\frac{1}{4}\end{pmatrix},\; C=\begin{pmatrix}0& \frac{1}{2\sqrt{3}}\\-\frac{1}{2\sqrt{3}}& 0\end{pmatrix}. 
\end{align}
Remark that $J$ is a \emph{circulant} block matrix. Further, $J$ is conjugate to the block diagonal matrix $(\frac{1}{2}R_{\pi/3}, \id_{\mathbb{R}^4})$ where $R_{\pi/3}:\mathbb{R}^2 \rightarrow \mathbb{R}^2$ is a rotation by $\pi/3$. Accordingly, there exists a $2$-dimensional subspace $E^s$ spanned by the eigenvectors $v_1,v_2$ corresponding to the two eigenvalues $\neq 1$. There exists $c > 0$ constant such that for any $v \in E^s_x$, $x \in \Lambda_{eq}$, one has 
$$\left\|Dt_xv\right\| \leq \frac{c}{2} \left\|v\right\|.$$
The four eigenvectors $v_3,\dots,v_6$ corresponding to eigenvalues $1$ span the tangent space of $\Lambda_{eq}$. So -- applying Corollary~\ref{c.convergence} -- the invariant set $\Lambda_{eq}$ is an \emph{attractor} for $\theta$. Hence, there exists a neighborhood $U_{eq} \supset \Lambda_{eq}$ such that every $x \in U_{eq}$ converges to $\Lambda_{eq}$ under iterates of $\theta$, that is, 
$$\dist(\theta^nx,\Lambda_{eq}) \rightarrow 0,\quad n \rightarrow \infty.$$
Taking into account Theorem~\ref{t.convergence} one concludes that $\Lambda_{eq}$ is a \emph{global attractor}.
\begin{rem}
If one considers the orbit space of the free group action $\isom(\mathbb{R}^2)$ on $\mathbb{R}^6$ by identifying similar triangles, one observes that this space is the two-dimensional projective space $P^2(\mathbb{R})$. By the observation above the triangle transformation passes to a well defined map on this quotient space:
\begin{equation*}
  \xymatrix{
\mathbb{R}^6 \ar[d]^p \ar[r]^t &\mathbb{R}^6\ar[d]^p\\
P^2(\mathbb{R}) \ar[r]^t &P^2(\mathbb{R})}
 \end{equation*}
 The attractor $\Lambda_{eq}$ projects to a globally attracting fixed point on $P^2(\mathbb{R})$. By the equivariance of the transformation $\theta$, the two-dimensional stable set tangent to $E^s$ passes also to a well-defined two-dimensional set in the orbit space reflecting the attraction of the fixed point.  
\end{rem}
\subsubsection{Case 2: mesh of six equilateral triangles}
Let $\Sigma=\left\{0,\dots,6\right\}$ and $C=\left\{(0,1,2), (0,2,3),(0,3,4),(0,4,5),(0,5,6),(0,6,1)\right\}$ be the connectivity and denote by $x_{eq}=(x_0,\dots,x_6)\in (\mathbb{R}^2)^7$ the mesh of six equilateral triangles. The group orbit $\Lambda_{eq}$ under the similarity group action is -- exactly as above -- a $4$-dimensional smooth submanifold of $\mathbb{R}^{14}$, and every mesh $x_{eq} \in \Lambda_{eq}$ is certainly a fixed point of the mesh transformation $\Theta$. We compute -- with the notations above -- the Jacobian matrix of $\Theta$ for $x_{eq}\in \Lambda_{eq}$ as  
\begin{align}\label{e.meshjacobian6}
&D\Theta(x_{eq})=\nonumber\\
&\begin{pmatrix}A & \frac{1}{6}(B+C)^T & \frac{1}{6}(B+C)^T & \frac{1}{6}(B+C)^T & \frac{1}{6}(B+C)^T & \frac{1}{6}(B+C)^T & \frac{1}{6}(B+C)^T \\\frac{1}{2}B & A &\frac{1}{2}B^T & 0 & 0 & 0 &\frac{1}{2}C^T\\ \frac{1}{2}B & \frac{1}{2}C^T & A &\frac{1}{2}B^T & 0 & 0 & 0\\ \frac{1}{2}B & 0 & \frac{1}{2}C^T & A &\frac{1}{2}B^T & 0 & 0 \\\frac{1}{2}B & 0 & 0 &\frac{1}{2}C^T & A &\frac{1}{2}B^T & 0 \\ \frac{1}{2}B & \frac{1}{2}B^T & 0 & 0 & \frac{1}{2}C^T & A & \frac{1}{2}B^T\\\frac{1}{2}B & 0 & 0 & 0 &0 & \frac{1}{2}C^T & A \end{pmatrix}.
\end{align} 
The matrix $D\Theta(x_{eq})$ has four eigenvalues $1$ whose eigenvectors span the $4$-dimensional tangent space of $\Lambda_{eq}$. Further, we have five pair of complex conjugate eigenvalues $\lambda_1,\overline{\lambda}_1,\dots, \overline{\lambda}_5$ with absolute values $\left|\lambda_i\right| \in [0.5774,0.8780]$. Consequently, the tangent space at $x_{eq} \in \Lambda_{eq}$ splits into a ten dimensional space $E^s(x_{eq})$ spanned by the eigenvectors $v_1,\overline{v}_1,\dots,\overline{v}_5$ and a $4$-dimensional eigenspace $T_{x_{eq}}\Lambda_{eq}$ of the equivalence relation: $$T_{\Lambda_{eq}}\mathbb{R}^{14} = E^s(\Lambda_{eq}) \oplus T\Lambda_{x_{eq}}.$$ So we can apply Corollary~\ref{c.convergence} and conclude that $\Lambda_{eq}$ is a local attractor, and consequently, there exists a neighborhood $\Lambda_{eq} \subset U_{eq} \subset \mathbb{R}^{14}$ such that every mesh $x=\mathcal{M}_C \in U_{eq}$ converges uniformly to one mesh $x_{eq}$ under $\Theta$: $$\dist(\Theta^n(x),x_{eq}) \rightarrow_{n \rightarrow \infty} 0 \quad x \in U_{eq}.$$ 
\begin{rem}
In contrast to the case of the triangle transformation we cannot prove that $\Lambda_{eq}$ is a global attractor: one observes numerically that $D\Theta(x)$ for $x \in X_C$ might have eigenvalues of absolute value $>1$, that is, there are directions in which $x$ is expanded. Numerical tests show, that after one or two iterations of $\Theta$, $x$ comes sufficiently close to $x_{eq}$ such that it converges uniformly to $x_{eq}$.  
\end{rem}
\subsubsection{Case 3: simple meshes}
let $\Sigma=\left\{0,\dots,N-1\right\}$ be a set of $N$ symbols. We call a connectivity $C$ $N$-\emph{simple} iff all triples $(i_0,i_1,i_2) \in C$ has a common symbol. We call $M_C$ a \emph{$N$-simple mesh} iff its connectivity $C$ is $N$-simple.\\
Above, we consider -- in this terminology -- a $6$-simple mesh. For a $N$-simple mesh, one fixed point of $\Theta$ is the mesh defined by the vertices $x_0=(0,0)$ and $x_{k-1}=(\cos(2k\pi/(N-1)), \sin(2k\pi/(N-1))$ for $k=2,\dots,N$. Let denote the similarity group orbit of this mesh by $\Lambda_{eq,N}$. 
We can then numerically compute the Jacobian matrix for $x_{eq}$ and show their spectra in Figure~\ref{fig:simplemesh} for $N=4,\dots,11$.
\begin{figure}[htpb]

\begin{minipage}{0.47\textwidth}
\includegraphics[width=\textwidth]{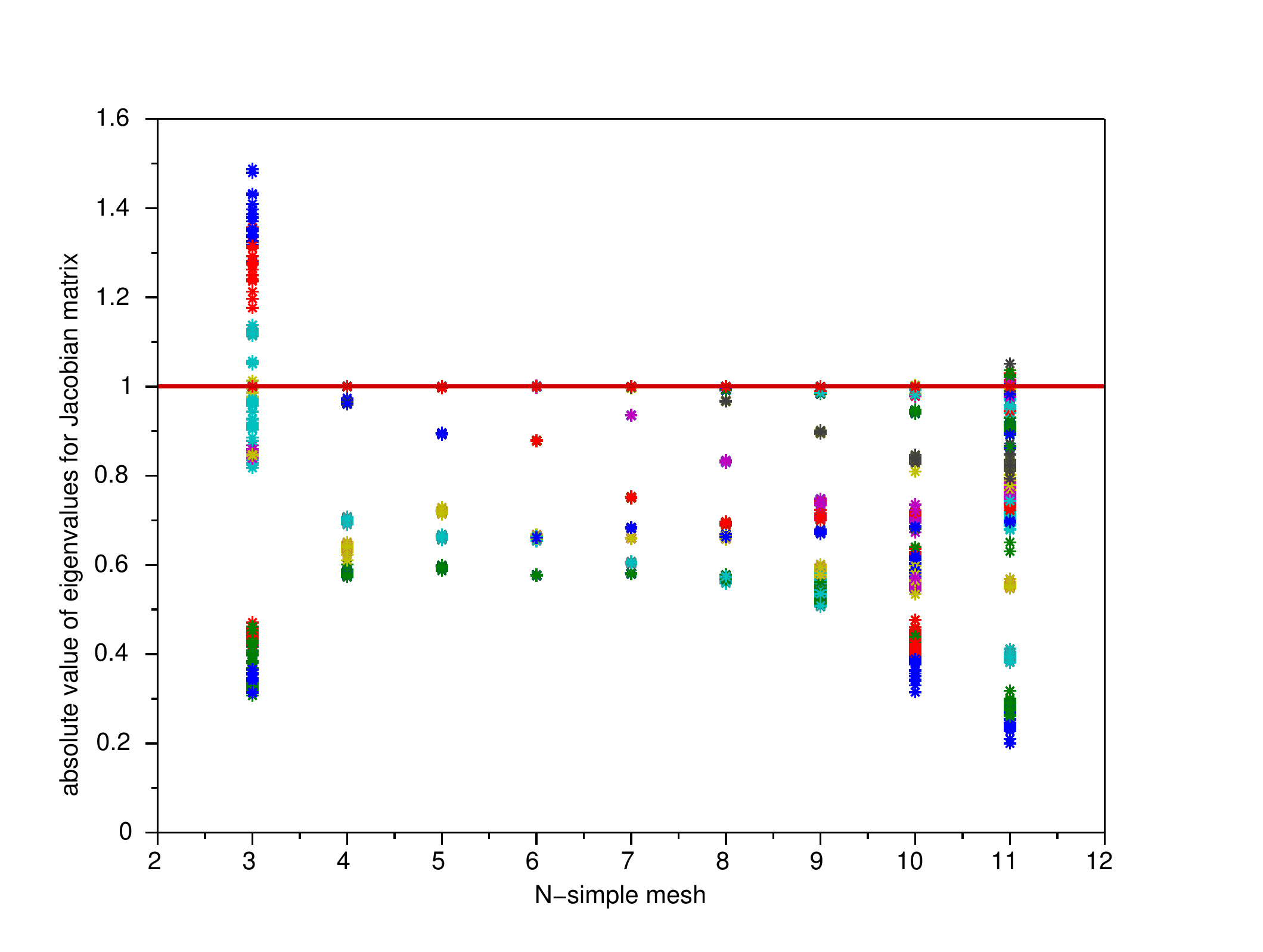}
\end{minipage}
\begin{minipage}{0.47\textwidth}
\includegraphics[width=\textwidth]{plot_simplemeshes-eps-converted-to.pdf}

\end{minipage}

\caption{Plot of the spectrum of $DT_N(x)$ where $T_N$ is the mesh transformation of a $N$-simple meshes and $x$ runs through 50 randomly generated $N$-simple meshes. On the right, the spectrum for the regular $N$-simple mesh is depicted. Note that the equilateral $3$-simple mesh is not an attracting point, but a \emph{saddle point}.}

\label{fig:simplemesh}
\end{figure}

We easily conclude that $\Lambda_{N,eq}$ is for $4 < N \leq 11$ a local attractor. Further, Figure~\ref{fig:simplemesh} seems to suggest that for $N \in [5,8]$, the fixed point set $\Lambda_{eq,N}$ is attracting in a quite large region.   

\subsubsection{Case 4: mesh of equilateral triangles}
Let $C$ be a connectivity such that every inner vertex has exactly six neighboring vertices, and consider the previously defined set $X_C$ of meshes with this connectivity. Denote by $N_i$ the indices of inner vertices and by $N_b$ the indices of boundary vertices. Let $x_{eq}=(x_0,\dots,x_{N-1})\in (\mathbb{R}^2)^N$ be the mesh of equilateral triangles. Exactly as above, we consider the whole group orbit $\Lambda_{eq} \subset X_C$ of the similarity group. Then we compute the Jacobian matrix $D\Theta(x_{eq})$ of the mesh transformation (\ref{e.mesh_transformation}) at $x_{eq} \in \Lambda_{eq}$: 
\begin{align}\label{e.meshjacobian}
D\Theta(x_{eq}) &= \left(\frac{\partial \Theta_k}{\partial x_l}\right)_{k,l=0,\dots,N-1}, \quad \frac{\partial \Theta_k}{\partial x_l} \in \mathbb{R}^{2 \times 2}\nonumber\\
\frac{\partial \Theta_k}{\partial x_l} &= A\quad \mbox{if} \; k=l\nonumber\\
\frac{\partial \Theta_k}{\partial x_l} &= \frac{1}{6}(B + C)^T \quad \mbox{if}\; l \in \Sigma_k,\; k\in N_i,\; \nonumber\\
\frac{\partial \Theta_k}{\partial x_l} &= \frac{1}{2}C \quad \mbox{if}\; l \in \Sigma_k,\; l,k\in N_b,\nonumber\\
\frac{\partial \Theta_k}{\partial x_l} &= \frac{1}{2}B \quad \mbox{if}\; l \in \Sigma_k,\; l\in N_i,\;k\in N_b,\;\nonumber\\
\frac{\partial \Theta_k}{\partial x_l} &=0 \quad \mbox{if}\; l\notin \Sigma_k.
\end{align} 
After various computations on different equilateral meshes we conjecture the following:
\begin{conj}
For any equilateral mesh $x$ the Jacobian matrix of $\Theta$ at $x$ has eigenvalues of absolute value $< 1$ except from exactly four. In particular, the group orbit $\Sim(\mathbb{R}^2).x$ of the mesh $x$ is an attractor.
\end{conj}  

\subsubsection{Further generalization}
One could again study the jacobian matrix of $\Theta$ at any fixed point $x$. But things get much more complicated, because the matrices could not be expressed in a simple way. 
We conjecture the following, where $\tilde{X}_C$ is the quotient space $X_C/\Sim(\mathbb{R}^2)$ of the group action of $\Sim(\mathbb{R}^2)$: 
\begin{conj}\label{t.global} For any $4 \leq N < \infty$ and any connectivity $C$ with cardinality $N$ the following is true: there exists a metric $\left\|\;\right\|_X$ on the quotient space $\tilde{X}_C$ such that the map $\tilde{\Theta}$ induced on $\tilde{X}_C$ is strictly contracting on its domain with respect to this metric, that is 
$$\left\|T(\mathcal{M}_C)-T(\mathcal{M}_C')\right\|_X \leq\lambda \left\|\mathcal{M}_C - \mathcal{M}'_C\right\|_X \;\mbox{for any two meshes}\; \mathcal{M}_C,\mathcal{M}_C' \in \tilde{X}_C.$$  
\end{conj}
This would imply in particular that any fixed point $\tilde{x} \in \tilde{X}_C$ is an attractor. 
\begin{rem} \begin{enumerate}

\item In Figure~\ref{fig:frobenius3} we show the absolute value of the six eigenvalues for 700 randomly generated triangles. This figure stresses also the fact that for not too distorted triangles the triangle transformation is strictly contracting transverse to the normally hyperbolic invariant set characterized by the four eigenvalues equal to one.

\item In Figure~\ref{fig:plot_meshtrans} we computed the norm of the Jacobian of the mesh transformation of a mesh of $7$ triangles (shown in the left picture) in relation to the matrix norm of the Jacobian for the most regular mesh of $7$ triangles. One observes in the right picture how the matrix norm approaches the optimal matrix norm as the quality of the triangle mesh approaches its optimum.   
\end{enumerate}
\end{rem}
\begin{figure}[htpb]
\centering
\begin{minipage}{0.47\textwidth}
\includegraphics[width=\textwidth]{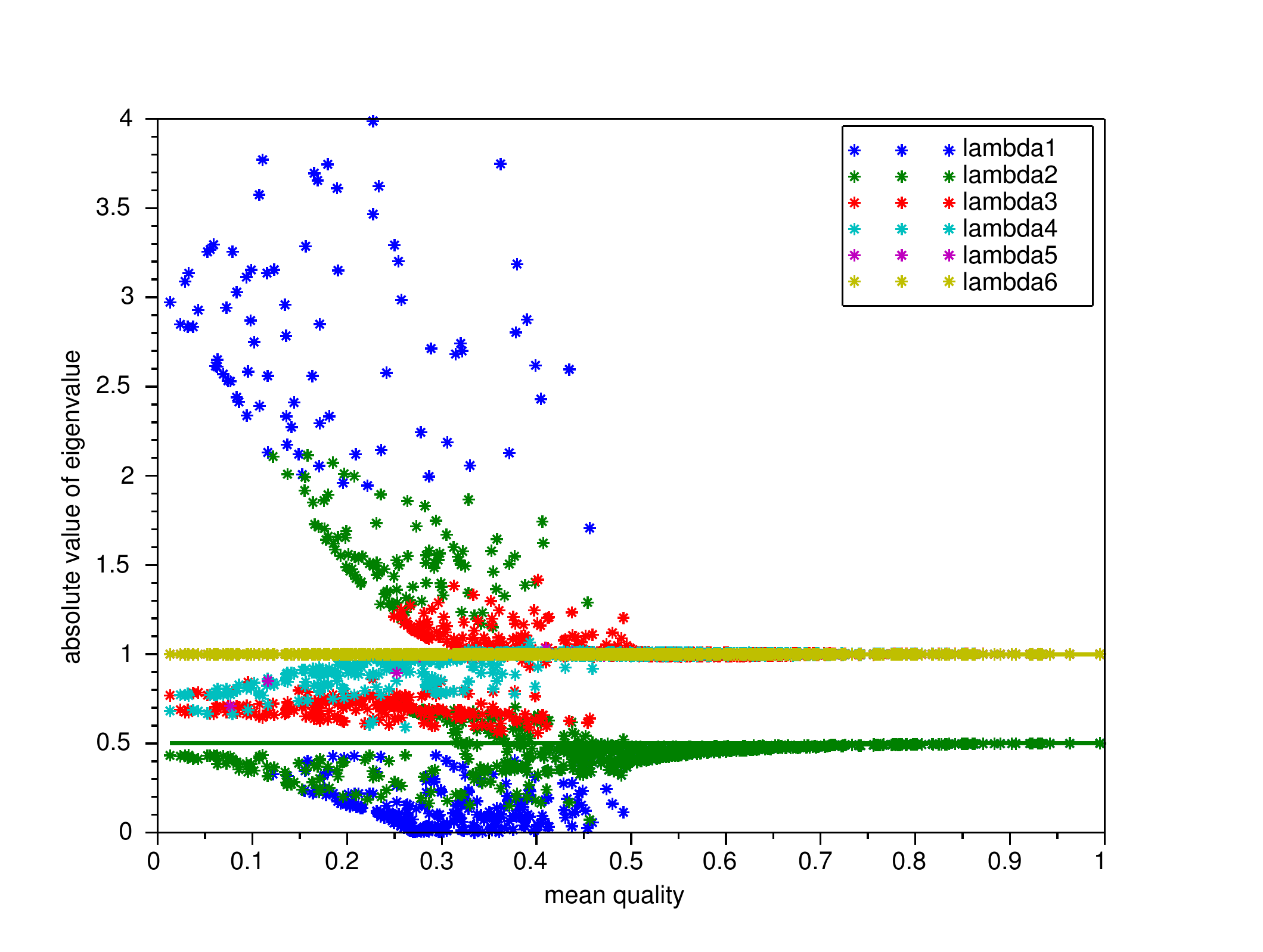}
\caption{Plot of the absolute value of eigenvalues in dependence of the triangle quality of 700 randomly generated triangles. }
\label{fig:frobenius3}
\end{minipage}
\centering
\begin{minipage}{0.47\textwidth}
\includegraphics[width=\textwidth]{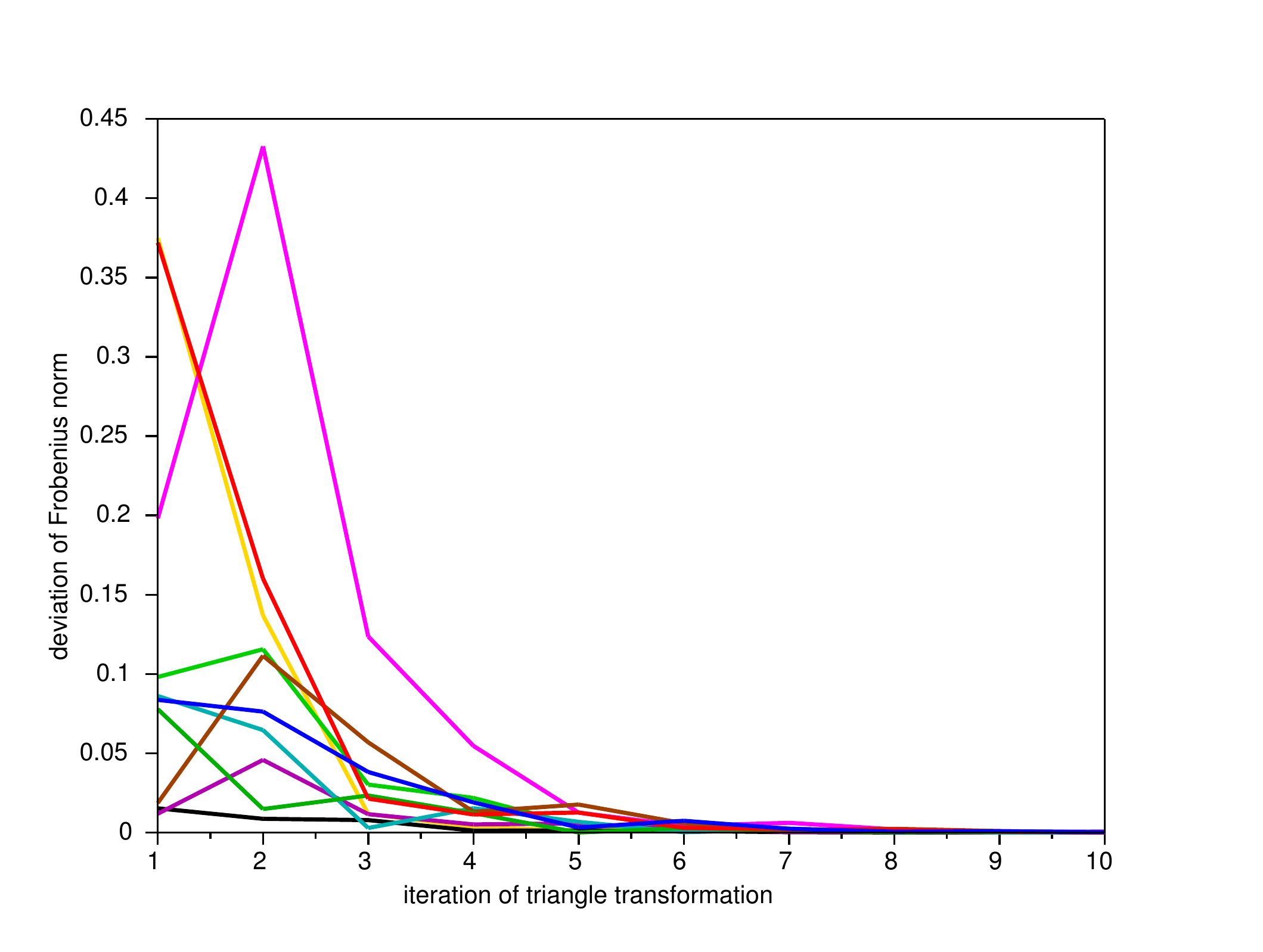}
\caption{Plot of the deviation $|\left\|dt^i\right\| - \left\|dt_e\right\||$ of the Frobenius norm of $t^i, i=1,\dots,10$ from the Frobenius norm $\left\|dt_e\right\|$ corresponding to an equilateral triangle.}
\label{fig:frobenius4}
\end{minipage}
\end{figure}

\paragraph{Outlook:}
The proof should be easily generalized for triangle meshes defined on Riemannian surfaces, that is, -- with the notations above -- the triangle mesh is defined by $M_C: \Sigma \rightarrow S, k \mapsto x_k \in S$ where $S$ is a Riemannian surface such that every triangle $M(\Delta)$ lies inside one chart neighborhood. \\
The techniques developed in this proof could also be adaptable to similar geometric mesh transformations. \\
In \cite{VB14}, we model the triangle transformation above by system of linear differential equations which could be seen as the description of coupled damped oszillations. This  model provides another explanation why the transformation converges to a equilateral triangle. One could think of the mesh transformation as the discretization of the solution of a system of coupled damped oszillations which are driven by each other antagonizing the damping. 

\begin{figure}[htpb]
\begin{minipage}{0.47\textwidth}
\includegraphics[width=\textwidth]{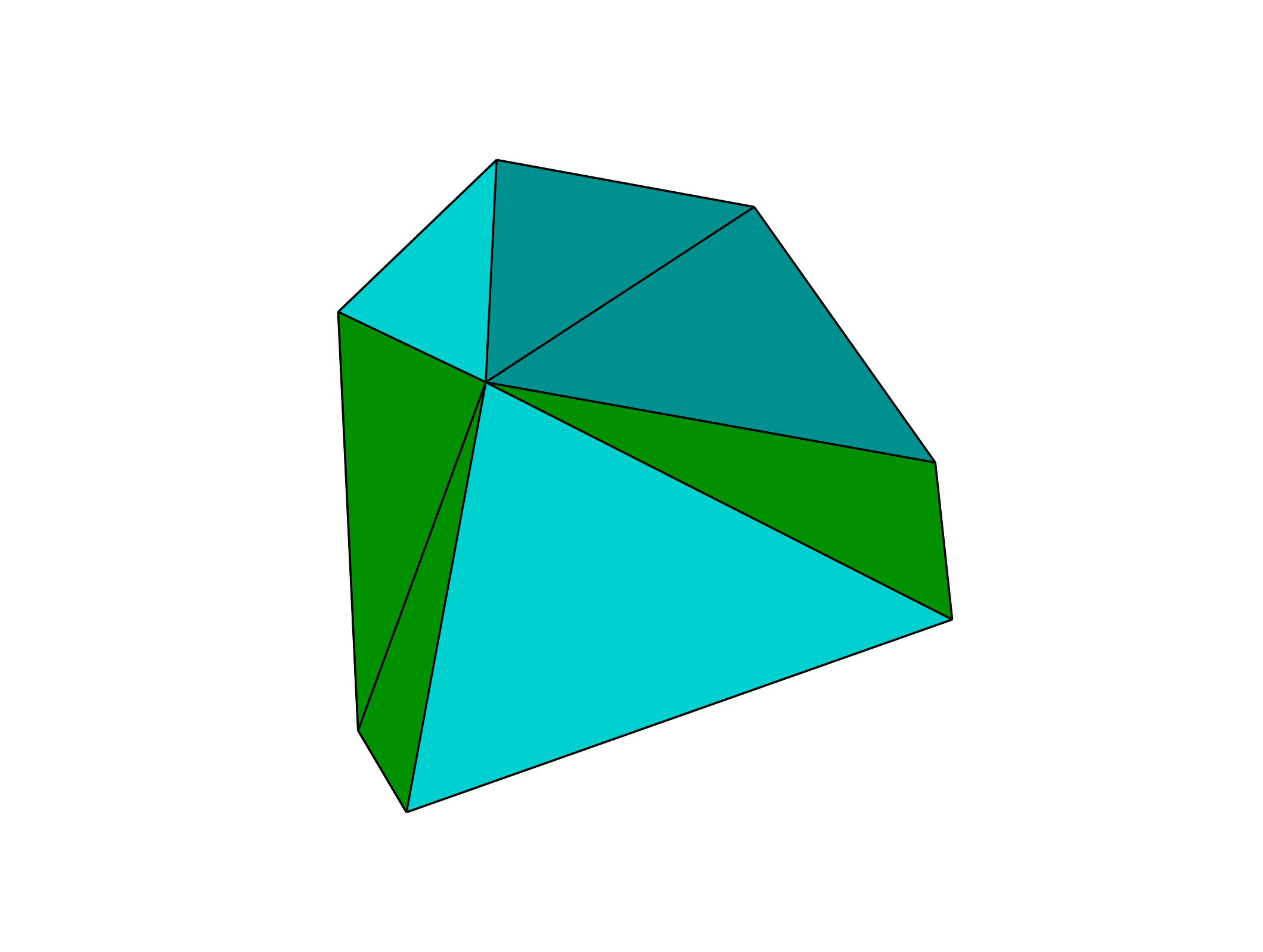}
\end{minipage}
\begin{minipage}{0.47\textwidth}
\includegraphics[width=\textwidth]{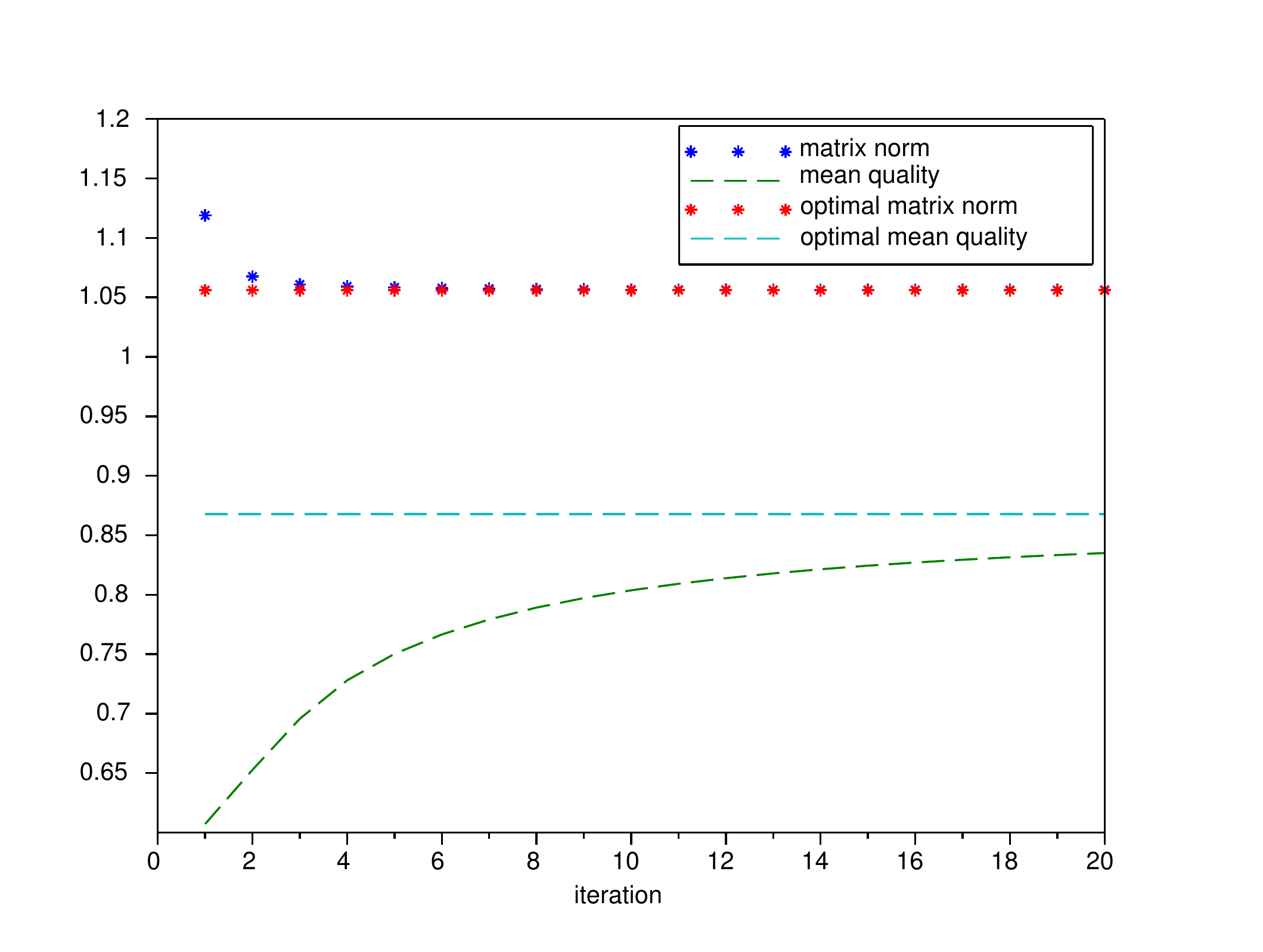}
\end{minipage}

\caption{Plot of $\left\|dT\right\|_2$ and mean mesh quality during iterations for a mesh of $7$ triangles together with the same values for an optimal mesh of $7$ triangles with inner angle $2\pi/7$. }
\label{fig:plot_meshtrans}
\end{figure}
\section{Short discussion of implementation and numerical results}\label{s.numerical}
We do not focus in this article on the application of our algorithm, so the following discussion is kept very brief and should be treated as a motivation to explore further the practical possibilities of the presented algorithm in the future. We have implemented the method as it is described above in Section~\ref{s.convergence} inside the open source software Scilab 5.4.1. The method could be equally well directly implemented in $C$. For an industrial usage this is strongly preferable to make it more efficient. \\
We tested the method for a randomly generated triangulation of the unit square. See below in Figure~\ref{fig:trimesh} how the mesh converges to a mesh of quite equilateral triangles in very few iterations. 
\begin{figure}[htbp]
\centering
\begin{minipage}{0.47\textwidth}
\includegraphics[width=\textwidth]{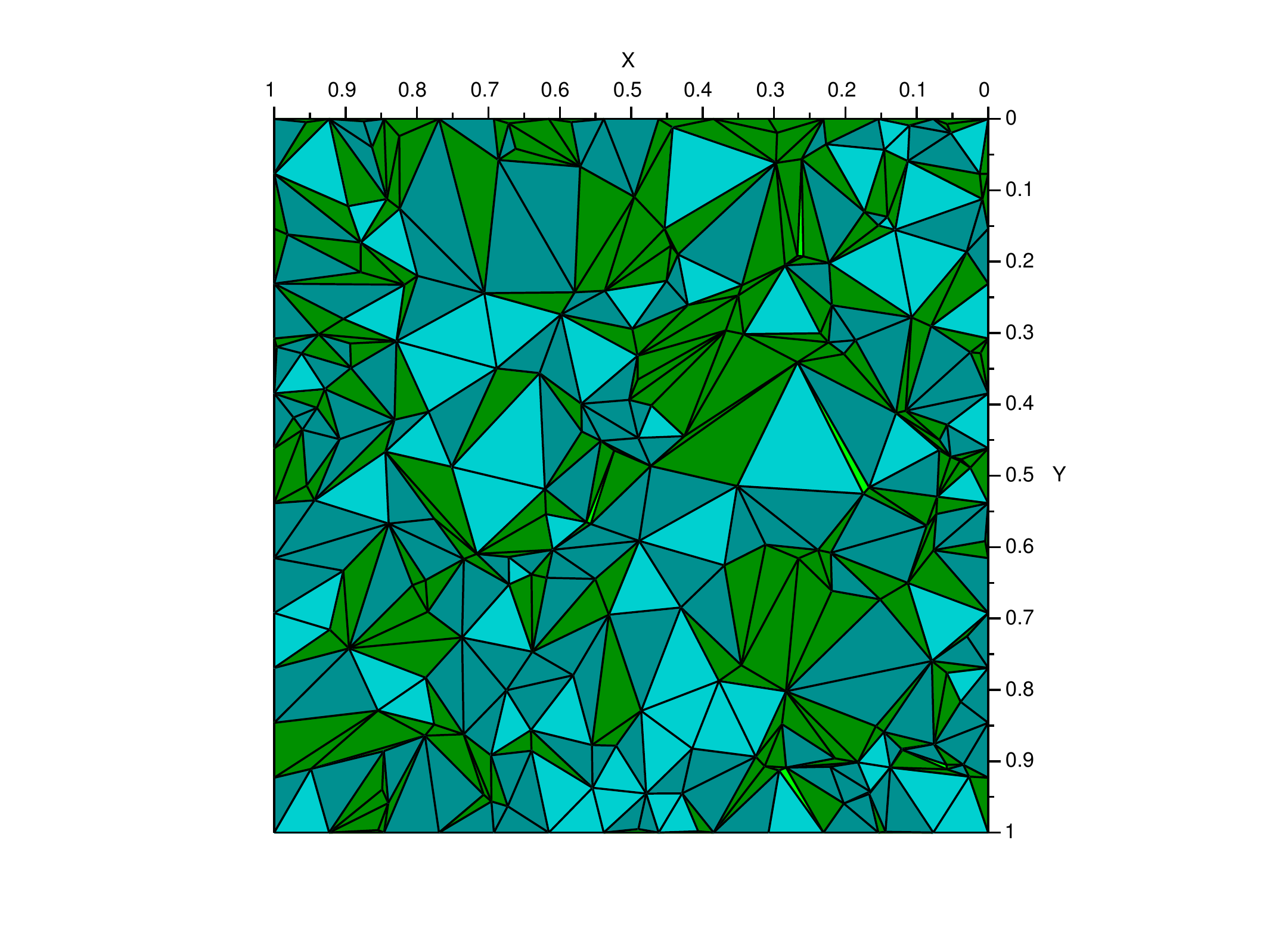}
\end{minipage} 
\begin{minipage}{0.47\textwidth}
\includegraphics[width=\textwidth]{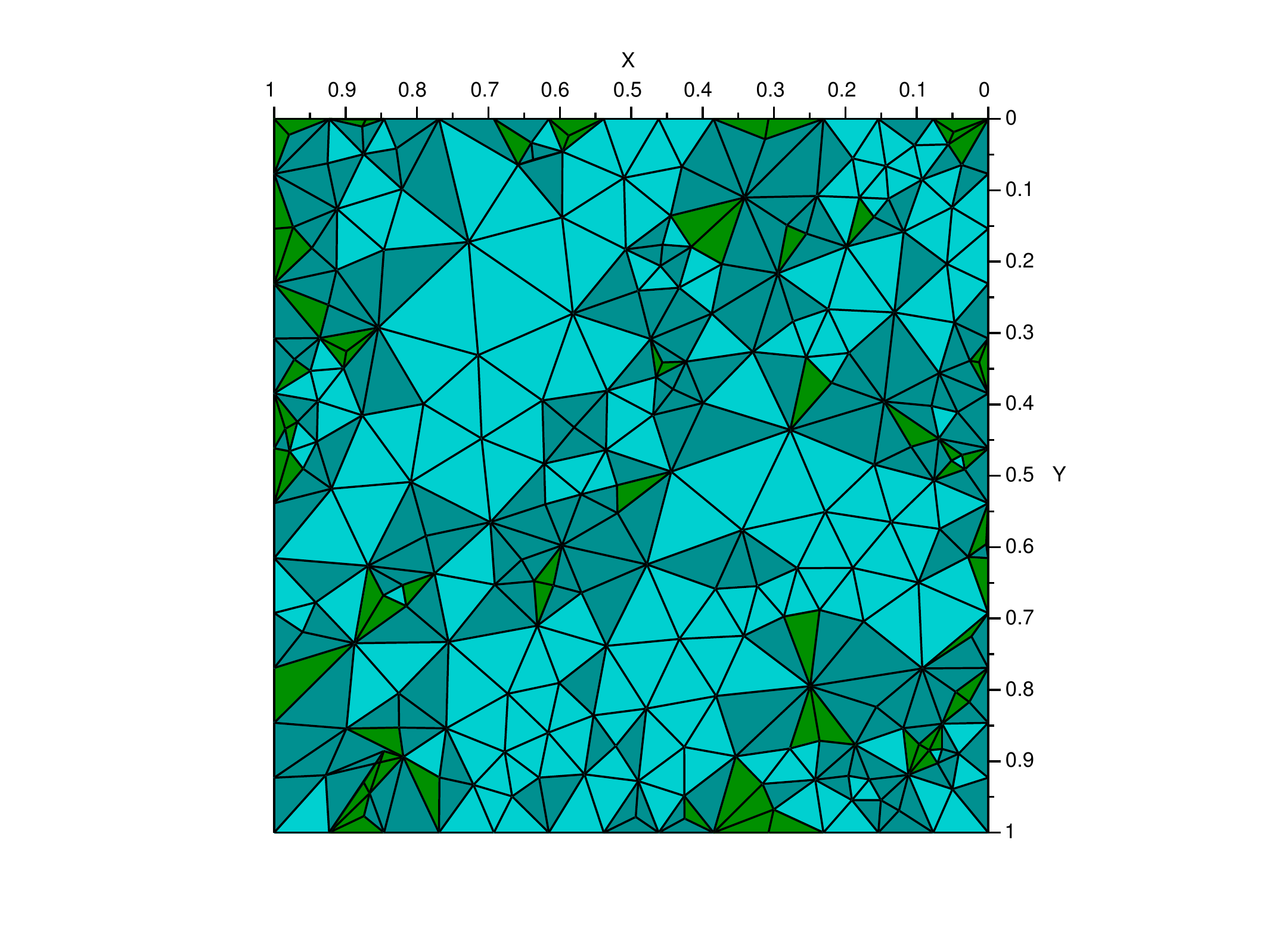}

\end{minipage}
\begin{minipage}{\textwidth}
\centering\vspace{0.2cm}
\includegraphics[width=8cm]{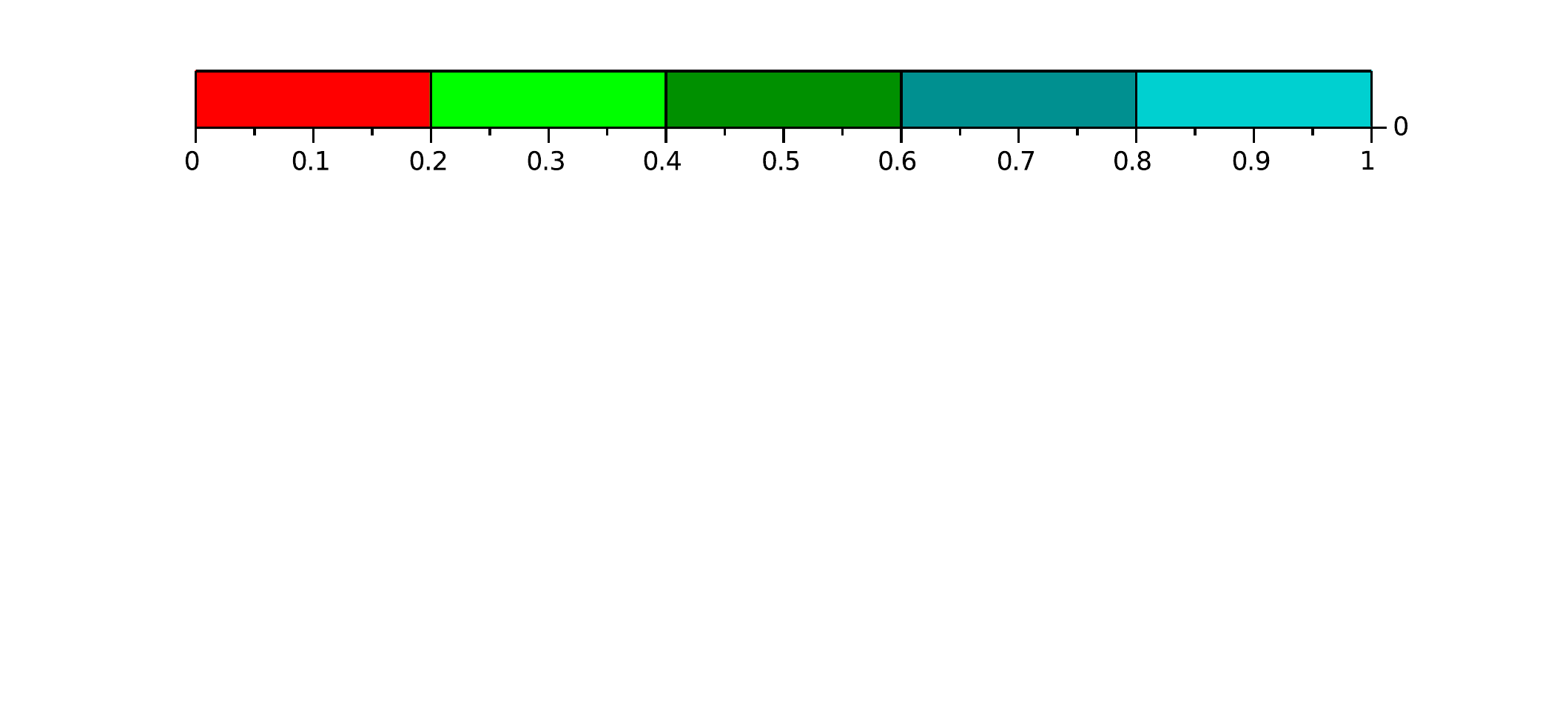}
\vspace{-2cm}

\end{minipage}
\caption{A randomly generated triangulation of the unit square at the beginning and after 10 iterations colored depending on their quality measure $q_{\Delta} \in (0,1]$.}
\label{fig:trimesh}
\end{figure} 
As quality measure $q_{\Delta}$ we used the ratio of minimal to maximal edge lengths of every triangle, $q_{\Delta}=\frac{\min_{i,j=1}^3 \left\|x_i-x_j\right\|}{\max_{i,j=1}^3 \left\|x_i -x_j\right\|}.$ The quality measure for a triangle mesh $V=(\Delta_0,\dots,\Delta_{|V|-1})$ is then the mean of the quality measure $q_{\Delta}$ for every triangle $\Delta \in V$: $q_V=\frac{1}{\left|V\right|}\sum_{\Delta \in V} q_{\Delta}.$\\
The mesh we smoothed in Figure~\ref{fig:trimesh} consists of 450 triangle elements. In Figure~\ref{fig:quality1} below we show how the number of elements with a certain quality measure develops over iterating the mesh and how the mean quality improves. \\
The smoothing algorithm works equally well for tetrahedra by applying the smoothing algorithm to the triangular faces. We display in Figure~\ref{fig:tetmesh} the cube $[0,1]^3$ cut at $x=0.5$ to show the improvement of the interior elements.
\begin{figure}[htbp]
\begin{subfigure}[l]{0.4\textwidth}
\includegraphics[width=\textwidth]{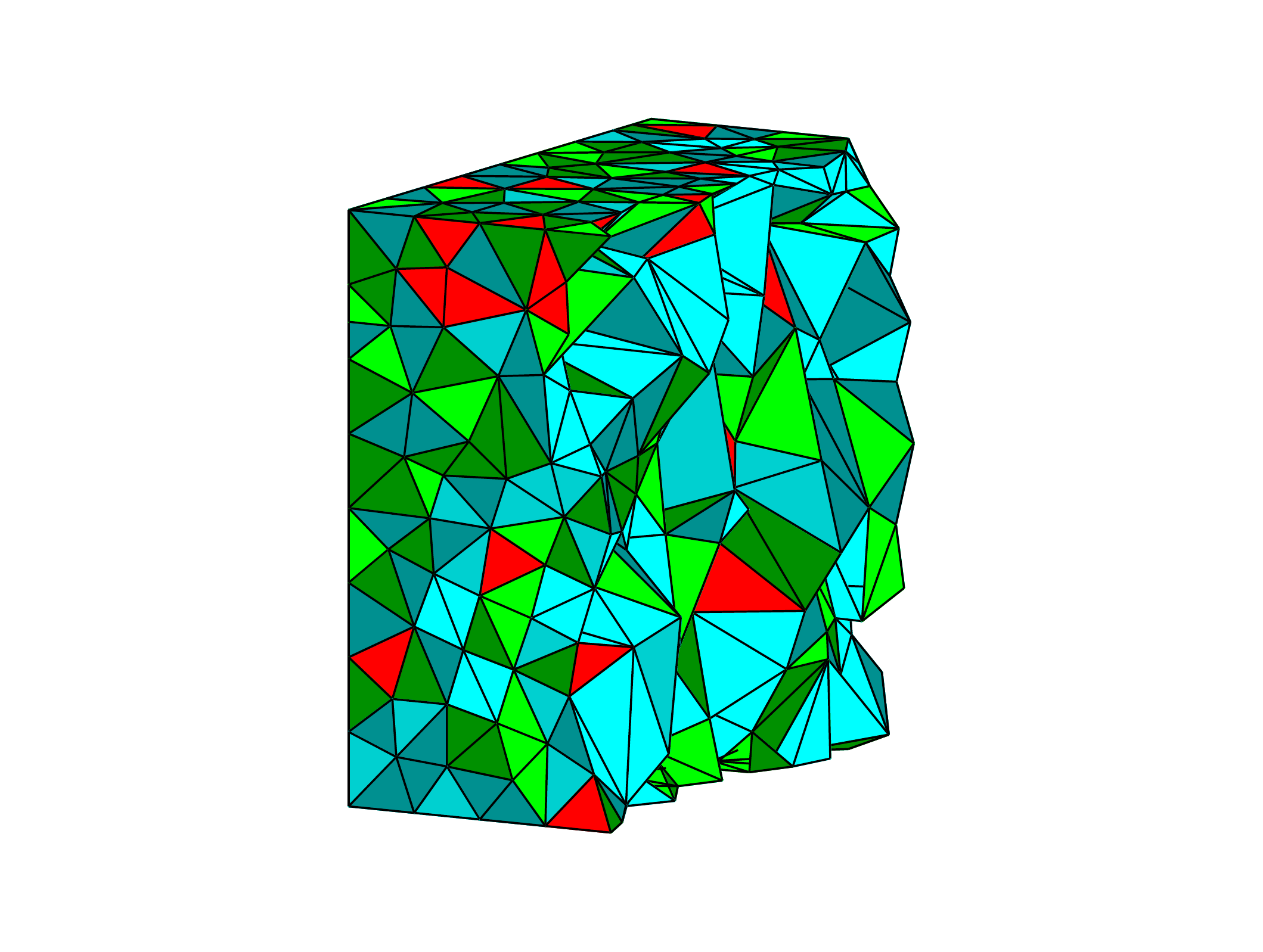}
\caption{Initial mesh: $q_V=0.4893$.}
\end{subfigure}
\begin{subfigure}[r]{0.4\textwidth}
\includegraphics[width=\textwidth]{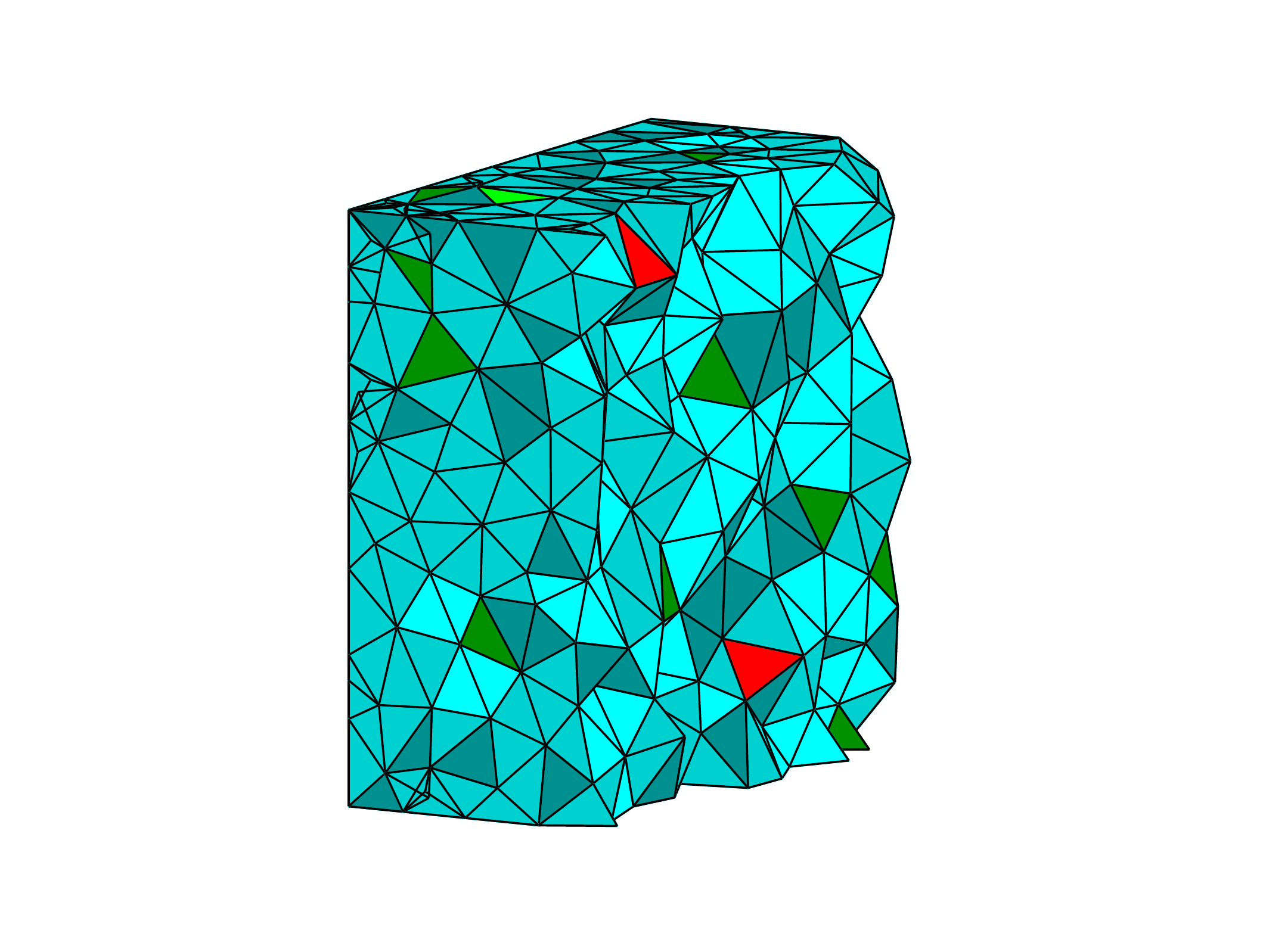}
\caption{10th iteration: $q_V=0.7652$.}
\end{subfigure}
\caption{Application to a tetrahedral mesh of 5318 elements of the unit cube (for $x < 0.5$).}
\label{fig:tetmesh}
\end{figure}

Let $T=(x_1,x_2,x_3,x_4)$ with $x_i \in \mathbb{R}^3$ be a tetrahedron. As quality measure $q_T$ for a tetrahedron $T$ we use the \emph{mean ratio quality measure} which is defined as following (see \cite{Knupp2001}):
\begin{align*}
q_T(T)&=\frac{3 \det(S)^{2/3}}{\trace(S^tS)}, \quad S=D(T)W, \;\mbox{where}\\
D(T)&=(x_2-x_1, x_3-x_1, x_4-x_1),\quad W=\begin{pmatrix}1& 1/2 &1/2\\0& \sqrt{3}/2&\sqrt{3}/6\\0&0&\sqrt{2/3}\end{pmatrix}. 
\end{align*}
As quality measure for a tetrahedral mesh $V=(T_0, \dots, T_{|V|-1})$ we used the mean quality measure of every element: 
$q_V=\frac{1}{|V|}\sum_{T \in V}q_T(T).$
In Figure~\ref{fig:quality}, one can observe how the quality measure of the mesh of Figure~\ref{fig:tetmesh} improves.
\begin{figure}[htpb]
\centering
\begin{subfigure}[l]{0.4\textwidth}
\includegraphics[width=0.9\textwidth]{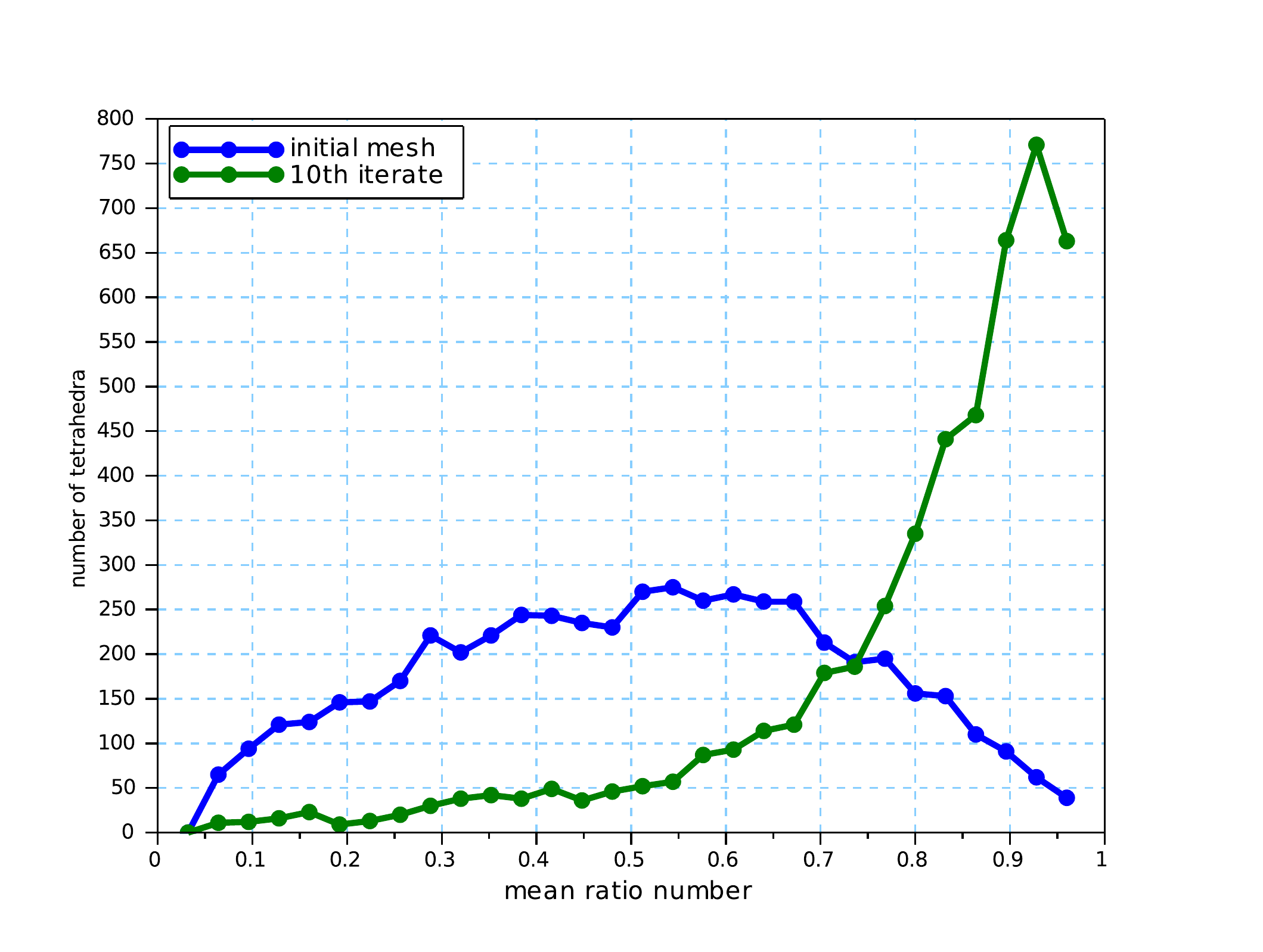}
\caption{Improvement of the cube mesh element quality for the mesh in Figure~\ref{fig:tetmesh}}
\label{fig:quality}
\end{subfigure}
\begin{subfigure}[r]{0.4\textwidth}
\includegraphics[width=0.9\textwidth]{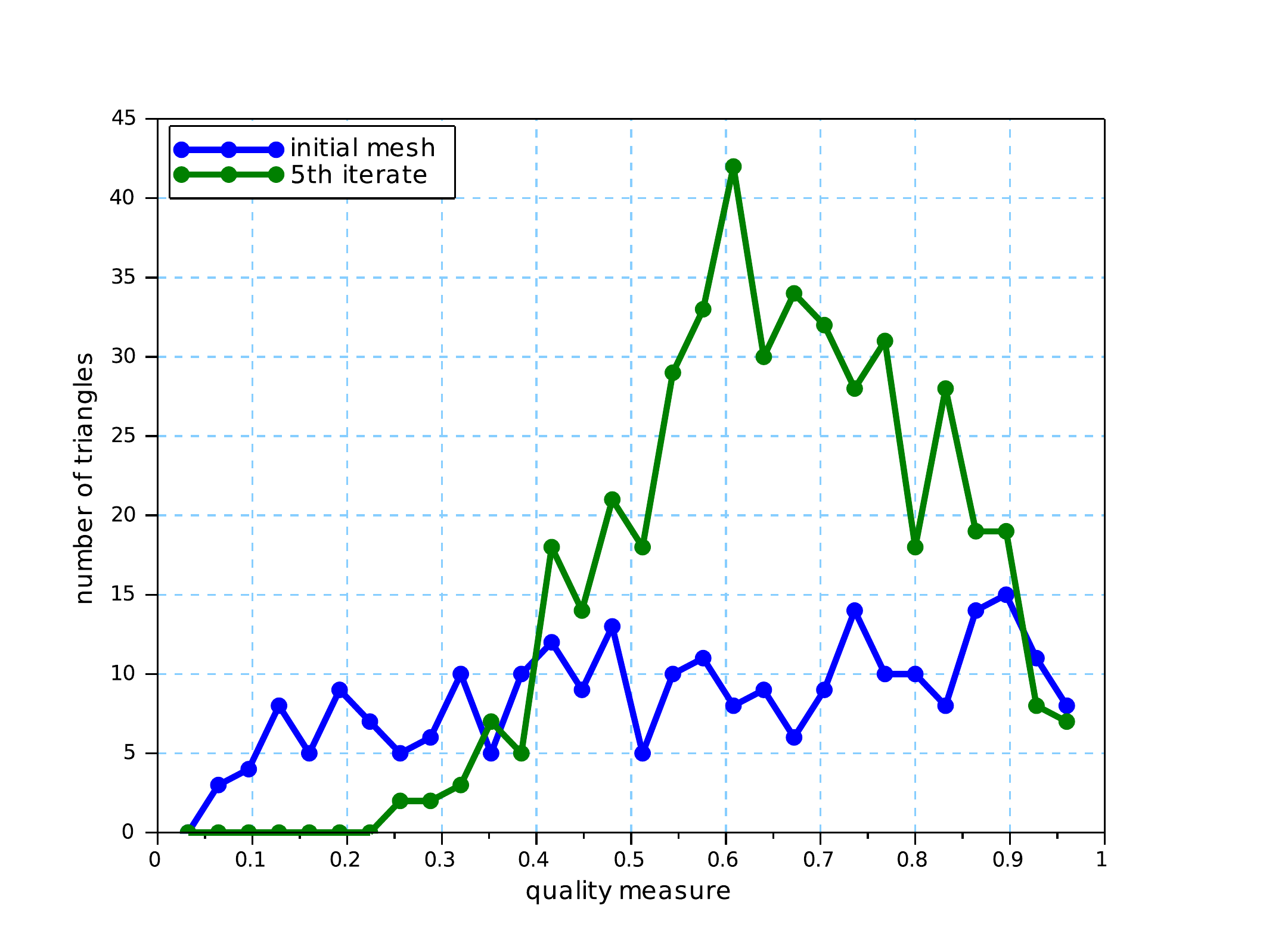}
\caption{Square quality element measure $q_{\Delta}$ before and after the smoothing of the mesh in Figure~\ref{fig:trimesh}.}
\label{fig:quality1}
\end{subfigure} 
\end{figure}

\section{Concluding remarks}
\subsection{Generalization to polygonal meshes}
Any polygon can be transformed in the exactly analogous way as the triangle above. Let $P=(x_0^{(0)},\dots,x_{k-1}^{(0)})$ be a convex $k$-gon with $x_i^{(0)} \in \mathbb{R}^2$ with its centroid in the origin. Then we can define a transformation in the following way recursively: 
$$x_i^{(n+1)}=\frac{k-1}{k}r_i^{(n)}x_i^{(n)} - \frac{1}{k}\sum_{j=0, j\neq i}^{k-1} r_j^{(n)}x_{j}^{(n)}, \quad r_i^{(n)}=\frac{\left\|x_{i-1}^{(n)}\right\|_2}{\left\|x_{i}^{(n)}\right\|_2}.$$
Remark that the centroid is kept in the origin througout the transformation. But the iterated polygon $P^{(n)}=(x_0^{(n)}, \dots,x_{k-1}^{(n)})$ does not necessarily converge for $n \rightarrow \infty$ to a polygon with equal distances $\left\|x_i\right\|_2=\left\|x_j\right\|_2$, $i,j=0,\dots,k-1$, i.e. its centroid coincides with its circumcenter. Consider for example a quadrilateral $Q^{(0)}$ with $\left\|x_0\right\|=\left\|x_2\right\|$ and $\left\|x_1\right\|=\left\|x_3\right\|$. Then $Q^{(0)}=Q^{(2)}$ is two-periodic but do not converge. So the transformation has not a globally attracting fixed point for all initial polygons. Also observe the following: while a triangle is regular if and only if the distances of its vertices to its centroid is equal, this is not the case for other polygons where it is just a necessary, but not a sufficient condition.

Accordingly, the transformation cannot be directly used for a smoothing algorithm for polygonal meshes without further adaption. \\
But nevertheless, the transformation can be used to smooth any polygonal mesh by subdividing every polygon into triangles and then applying the transformation to every triangle.   
\subsection{Outlook}
One easily detects the following shortcomings of the presented smoothing method which are open for future research: 
\begin{description}
\item[Global convergence] We only prove the global convergence for a compact subset of triangle meshes which exclude triangles close to degenerate ones. By changing the transformation a bit -- with regard to the estimates we derive during the proof -- such that the transformation is integrable, that is, the gradient of a function, and consequently the jacobian matrix normal, one could obtain better bounds and therefore extend the convergence result to a greater subset of meshes. 
\item[Performance] It was not the primary objective of this article to provide an efficient implementation, but to analyze the underlying mathematics. So we have to admit that each iteration step is numerically quite long in the present implementation. But if directly implemented inside $C$, we should attain comparable run times as for GETMe. As the algorithm relocates separately every vertex, it is open for an application of parallel computing techniques.
\item[Polygonal/ polyhedral meshes] Using the duality of certain polygons/ polyhedra to each other  we hope to be able to adapt the current transformation to quadrilateral and hexahedral meshes. This is a current topic of our research.  
\end{description}

\bibliographystyle{plain}
\bibliography{bibfile}
\end{document}